\documentclass[oneside,12pt]{amsart}
\usepackage{amssymb, amscd, mathrsfs}
\usepackage[all]{xy}
\usepackage{enumerate}

\newtheorem{thm}{Theorem}[section]
\newtheorem{cor}[thm]{Corollary}
\newtheorem{lem}[thm]{Lemma}
\newtheorem{defn}[thm]{Definition}
\newtheorem{prop}[thm]{Proposition}

\newtheorem{conj}[thm]{Conjecture}

\newtheorem{remk}[thm]{Remark}

\newcommand{\del}[2]{{}}

\newcommand{\Z}{\mathbb Z}

\newcommand{\mF}{\mathfrak {F}}
\newcommand{\mG}{\mathfrak {G}}

\newcommand{\mX}{\mathfrak {X}}
\newcommand{\cd}{{cd}_{\mF}}

\newcommand{\hix}[1]{\mbox{${\scriptstyle\bf H}^{\mathfrak F}_{#1}\mathfrak X$}}

\newcommand{\hf}[1]{\mbox{${\scriptstyle\bf H}_{#1}\mathfrak F$}}
\newcommand{\hff}[1]{\mbox{${\scriptstyle\bf H}^{\mathfrak F}_{#1}\mathfrak F$}}

\newcommand{\orb}{\mathcal{O}_{\mF}G}
\newcommand{\orbs}{\mathcal{O}_{\mF\cap S}S}
\newcommand{\orbmod}{\mbox{Mod}_{\mF}G}
\newcommand{\orbmods}{\mbox{Mod}_{\mF\cap S}S}
\newcommand{\gorbmod}{G\mbox{-}\mbox{Mod}_{\mF}}

\newcommand{\zmod}{\mathfrak{Ab}}
\newcommand{\nathom}{\mathrm{Hom}_{\mF}}

\title{Intermediaries in Bredon (Co)homology and Classifying Spaces}

\author{Fotini Dembegioti}
\address{Department of Mathematics, University of Athens, Athens, Greece}%
\email{fdebeg@math.uoa.gr}%

\author{Nansen Petrosyan}
\address{Department of Mathematics, Katholieke Universiteit Leuven, Kortrijk, Belgium}%
\email{Nansen.Petrosyan@kuleuven-kortrijk.be}%

\author{Olympia Talelli}
\address{Department of Mathematics, University of Athens, Athens, Greece}%
\email{otalelli@math.uoa.gr}%

\date{\today}

\begin{document}
\begin{abstract}
For certain contractible $G$-CW-complexes and $\mathfrak F$ a family of subgroups of $G$, we construct a spectral sequence converging to the $\mathfrak F$-Bredon cohomology of $G$ with $\mathrm{E}_1$-terms given by the $\mathfrak F$-Bredon cohomology of the stabilizer subgroups. As applications,  we obtain several corollaries concerning the cohomological and geometric dimensions of the classifying space $E_{\mF}G$. We also introduce, for any subgroup closed class of groups $\mathfrak F$, a hierarchically defined class of groups which contains all countable elementary amenable groups and countable linear groups of characteristic zero, and show that if a group $G$ is in this class, then $G$ has finite $\mathfrak F\cap G$-Bredon (co)homological dimension if and only if $G$ has jump  $\mathfrak F\cap G$-Bredon (co)homology.
\end{abstract}

%\maketitle
%\tableofcontents

\maketitle

\section{Introduction}

In \cite{Bredon}, Bredon introduced a  cohomology theory for equivariant CW-complexes. One of the motivations was to develop an obstruction theory for equivariant extension of maps. This cohomology was further studied by many authors with applications to proper group actions and classifying spaces for  families of subgroups (e.g. see L\"{u}ck \cite{Luck2}).

%Let $\orbmod$ be the category of contravariant functors from the orbit category $\mathcal{O}_{\mF}G$ to the category of abelian groups $\mathfrak{Ab}$. The Bredon cohomology functors $\mathrm{H}_{\mF}^{\ast}(G,-)$ of a group $G$ with respect to the family of subgroups $\mF$ are the right derived functors of a certain $\mathrm{Hom}$ functor from  $\orbmod$ to $\mathfrak{Ab}$.

For a collection of subgroups $\mF$ of a group $G$ that is closed under conjugation and finite intersections, one can consider the homotopy category of $G$-$\mathrm{CW}$-complexes with stabilizers in $\mF$. A terminal
object in this category is called a model for the classifying space $E_{\mF}G$, and one can show that these models always exist under the given assumptions on $\mF$ (e.g. see \cite{Luck2}). The  augmented cellular chain complex of $E_{\mF}G$ yields a projective resolution which can be used to compute $\mF$-Bredon cohomology of $G$. In relation to Baum-Connes and Farell-Jones Isomorphism conjectures (see \cite{BCH}, \cite{MV}, \cite{DL}), there has been extensive research on questions concerning finiteness properties of $E_{\mF}G$, in particular when $\mF$ is the family of finite subgroups or the family of virtually cyclic subgroups (e.g. see \cite{KMN},\cite{luck1},\cite{LW},\cite{nuc}).

The purpose of this paper is  twofold. First, given a $G$-CW-complex $X$  such that for each  $F\in \mF$ the subcomplex $X^F$ is contractible, we construct a spectral sequence converging to the $\mF$-Bredon cohomology of $G$.  The $\mathrm{E}_1$-terms of the spectral sequence are given in terms of Bredon cohomology of the stabilizer subgroups (see 3.5). As applications, we obtain several corollaries concerning the cohomological and geometric dimensions of $E_{\mF}G$, some of which are refinements of known results (e.g. Corollaries \ref{main5}, \ref{sc:08}, \ref{main7}).

Secondly, using a construction similar to that of Kropholler's class ${{\scriptstyle\bf H}}\mathfrak{X}$ (see \cite{krop}), we introduce a new class of groups which is quite suited for applying this spectral sequence.

%This is not a strong assumption on the CW-complex, because by a subdivision, we can assume that any $G$-CW-complex will satisfy this condition?

\begin{defn}{\label{introsc:01}}{\normalfont Let $\mF$ be a  class of groups closed under taking subgroups and let $G$ be a group. We set $\mF\cap G=\{H\leq G \; | \;  \mbox{$H$ is isomorphic to a subgroup in $\mF$}\}.$   Let $\mX$ be a class of groups. Then \hix{} is defined as the smallest class of groups containing the class
$\mX$ with the property that if a group $G$ acts cellularly on a
finite dimensional CW-complex $X$ with all isotropy
subgroups in \hix{} and such that for each subgroup $F\in \mF\cap G$ the fixed point set
$X^F$ is contractible, then $G$ is in \hix{}. ${{\scriptstyle\mathrm{L}\bf H}}^\mathfrak{F}\mX$ is defined to be the class of groups that are locally \hix{}-groups.}
\end{defn}

If $\mF$ contains the trivial group only, then \hix{} is exactly Kropholler's class  ${{\scriptstyle\bf H}}\mathfrak{X}$. Perhaps surprisingly, when $\mF$ is the class of all finite groups, then ${{\scriptstyle\bf H}}^{\mathfrak{F}}\mathfrak{F}$ also turns out to be quite large. It contains all countable elementary amenable groups and all countable linear groups over a field of characteristic zero. It is also closed under extensions, taking subgroups, amalgamated products, HNN-extensions, and countable directed unions.
In view of this, we propose the following
\begin{conj}{\rm Let $\mF$ be the class of finite or virtually cyclic groups. Then, ${{\scriptstyle\bf H}}\mathfrak{F}={{\scriptstyle\bf H}}^{\mathfrak{F}}\mathfrak{F}$}.
\end{conj}
Next, we extend the notion of jump (co)homology to the context of Bredon  (co)homology
%\begin{thm}{\label{introsc:05}}{\normalfont Let $\mF$ be a subgroup closed class of groups and let $\mX$ be a class of groups with finite Bredon (co)homological dimension with respect to $\mF$.
%Let $G$ be a group in ${{\scriptstyle\mathrm{L}\bf H}}^\mathfrak{F}\mX$. Then $G$ has finite Bredon (co)homological dimension  if and only if $G$ has jump Bredon (co)homology with respect to $\mF$.}
%\end{thm}
and prove

 \smallskip
\noindent \textbf{Theorem} (Theorem \ref{sc:05}, Remark
\ref{sc:05.6}) {\normalfont A ${{\scriptstyle\mathrm{L}\bf
H}}^\mathfrak{F}\mF$-group $G$ has jump $\mF$-Bredon cohomology if
and only if $G$ has a finite dimensional model for ${E}_{\mF}G$,
where $\mF$ is the class of finite groups or the class of
virtually cyclic groups.}

\smallskip Based on this, we propose
\begin{conj}{\rm Let $G$ be a ${{\scriptstyle\mathrm{L}\bf H}}\mathfrak{F}$-group. Then $G$ has jump $\mF$-Bredon cohomology if and only if $G$ has a finite dimensional model for ${E}_{\mF}G$, where $\mF$ is the class of finite groups or the class of
virtually cyclic groups.}
\end{conj}

Lastly, we address the applications of our results to elementary amenable groups. In particular, we show

\smallskip
\noindent \textbf{Theorem} (Theorem \ref{sc:11})
 If an elementary amenable group $G$ acts freely and properly discontinuously on a manifold homotopy equivalent to a closed manifold, then $G$ has a finite dimensional model for $\underline{E}G$.

\section{Notations and Preliminaries}

We begin with some notions and well-established facts concerning modules over the orbit category. To do homological algebra in the category of Bredon modules and in particular to define Bredon (co)homology, we follow the approach of L\"{u}ck in \cite{luck1}, Chapter $9$.  Let us briefly recall some of the elementary notions we will need.

Let $G$ be a group and let $\mF$ be a collection of subgroups of $G$ closed under conjugation by every element of $G$ and closed under taking subgroups.  Then for every subgroup $S$ of $G$, the collection $\mF\cap S=\{H\cap S \; |\; H\in \mF\}$ is contained in $\mF$. We refer to $\mF$ as a \emph{family of subgroups} of $G$.

Let  $\orb$ denote the orbit category over $\mF$.  A  \emph{right $\orb$-module} is a contravariant functor $M: \orb \rightarrow \zmod$. The category $\orbmod$  is defined as follows: the objects are all the $\orb$-modules and the morphisms are all the natural transformations between the objects. Similarly, a \emph{left $\orb$-module} is a covariant functor $N: \orb \rightarrow \zmod$ and $\gorbmod$  is the category with objects the left $\orb$-modules and morphisms all the natural transformations between them.

For $M\in\orbmod$ and $N\in\gorbmod$, we can define their tensor product over $\mF$ by:
$$M\otimes_{\mF}N= \bigoplus_{H\in \mF}M(G/H)\otimes N(G/H)/\thicksim$$
where we identify $M(\varphi)(m)\otimes n \thicksim m \otimes N(\varphi)(n)$ for each morphism ($\varphi:G/H\to G/K)\in \orb$ and elements $m\in M(G/K)$, $n\in N(G/H)$.

A sequence $ 0\to M' \rightarrow M \rightarrow M''\to 0$
in $\orbmod$ is called {\it exact} if it is exact after evaluating in $G/H$, for all $H \in \mF$. Now, let $M \in \orbmod$ and consider the left exact functor:
\begin{align*}
 \nathom(M,-) : \orbmod &\rightarrow \zmod \\
N &\mapsto \nathom(M,N),
\end{align*}
where $\nathom(M,N)$ is the abelian group of all natural transformations from $M$ to $N$.
The $\orb$-module $P$ is  said to be {\it projective} if the functor  $\nathom(P,-)$ is exact. Just as for $\Z G$-modules, projective $\orb$-modules are flat, i.e. the functor $P\otimes_{\mF}-$ is exact.

The {\it free} $\orbmod$-module $P^G_K=\mathbb{Z}[\mbox{map}_G(-,G/K)]$ {\it based} at $G/K\in \orb$ assigns to each $G/H\in \orb$ the free $\Z$-module generated by the set of all $G$-maps from $G/H$ to $G/K$. In general, an $\orb$-module is said to be {\it free} if it is naturally isomorphic to a direct sum of free $\orb$-modules based at objects of $\orb$. An $\orb$-module $P$ is projective if and only if it a direct  summand of a free module if and only if every exact sequence $0\to M\to N\to P\to 0$ splits (see \cite[9.20]{luck1}).

It can be shown that $\orbmod$ contains enough projective  objects to construct projective resolutions. Hence, we can define  \emph{$n$-th Bredon homology} of $G$ with coefficients in $N \in \gorbmod$ to be:
\[ \mathrm{H}_n^{\mF}(G,N)= \mathrm{Tor}_n^{\mF}(\underline{\Z}, N)=\mathrm{H}_{n}(P_{\ast}\otimes_{\mF} N), \]
where $P_{\ast}\twoheadrightarrow\underline{\Z}$ is a projective $\orb$-resolution of the trivial $\orb$-module $\underline{\Z}$. Also, \emph{$n$-th Bredon cohomology} of $G$ with coefficients in $M \in \orbmod$ can be defined as:
\[ \mathrm{H}^n_{\mF}(G,M)=\mathrm{Ext}^n_{\mF}(\underline{\Z}, M)= \mathrm{H}^{n}(\nathom(P_{\ast},M)).\]
Given $M, N\in \orbmod$, we denote by $M\otimes N$ the right $\orb$-module which assigns to each $G/H\in \orb$ the $\Z$-module $M(G/H)\otimes N(G/H)$ and to each morphism $\varphi:G/H\to G/K$ of $\orb$, the homomorphism: $$M(\varphi)\otimes N(\varphi):M(G/K)\otimes N(G/K)\to M(G/H)\otimes N(G/H).$$

For each subgroup $S$ of $G$, let $\mathcal{I}:\orbs\to \orb$ be the covariant inclusion functor:
\begin{align*}
S/H &\buildrel{\mathcal{I}^{ob}}\over{\longmapsto} G/H,\\
\smallskip
(S/H {\buildrel{x}\over{\to}} S/K : H\mapsto xK)&\buildrel{\mathcal{I}^{mor}}\over{\longmapsto}(G/H {\buildrel{x}\over{\to}} H/K : H\mapsto xK).
\intertext{As for modules over group rings, associated to $\mathcal{I}$, there are induction and restriction (additive) functors:}
\mbox{Res}_{S}^{G}: \orbmod &\to \orbmods \\
M&\mapsto M\circ \mathcal{I},\\
\medskip
\mbox{Ind}_S^{G}: \orbmods &\to \orbmod \\
M&\mapsto M(?)\otimes_{\mF \cap S} \mathbb{Z}[\mbox{map}_G(-,\mathcal{I}(?))].
%\medskip
%\mbox{Coind}_S^{G}: \orbmods &\to \orbmod \\
%M&\mapsto  {\mathrm{Hom}_{\mF}(\mathbb{Z}[\mbox{map}_G(?,F(-))], M(?)).}
\end{align*}

We will need the following relations between them.

\begin{lem}[\cite{mar}, 3.5]{\label{ind}}{\normalfont Let $\mF$ be any given family of subgroups of $G$. Suppose $S$ is a subgroup of $G$, the modules $M, Q\in $ Mod$_\mF G$ and $N\in G$-Mod$_{\mF}$. Then, there are natural isomorphisms:
\begin{align*}
(Q\otimes \mbox{Ind}_{S}^{G}\underline{\Z})\otimes_{\mF}N &\cong \mbox{Res}_{S}^{G}Q\otimes_{\mF\cap S} \mbox{Res}_{S}^{G}N,\\
\mbox{Hom}_{\mF}(Q\otimes \mbox{Ind}_{S}^{G}\underline{\Z}, M)&\cong \mbox{Hom}_{\mF\cap S}(\mbox{Res}_{S}^{G}Q, \mbox{Res}_{S}^{G}M).
\end{align*}}
\end{lem}

\section{Spectral sequence}

In this section, we  construct a spectral sequence associated to a $G$-CW-complex which is the Bredon analog of an $\mathrm{E}_1$-term spectral sequence
converging the (co)homology of $G$ (see \cite[p. 173-174]{brown}).

\begin{lem}{\label{proj}}
{\normalfont Let $\mF$ be any given family of subgroups of $G$.  For $F\in \mF$ and a subgroup $H$  of $G$, the module $\mbox{Res}_{F}^{G}\mbox{Ind}_{H}^{G}\underline{\Z}$ is projective in $\mbox{Mod}_{\mF\cap F}F$.}
\end{lem}
\begin{proof} By \cite[3.3]{mar}, there is an isomorphism, $\mbox{Ind}_{H}^{G}\underline{\Z}\cong \Z[\mbox{map}_G(-, G/H)]$. We claim that the module
$M=\mbox{Res}_{F}^{G}\Z[\mbox{map}_G(-, G/H)]$ is a direct summand of the free  ${\mathcal{O}_{\mF\cap F}F}$-module $N=\bigoplus_{x\in G}{\Z}[\mbox{map}_G(-, G/H\cap F^x)]$.
To prove this, we  define natural transformations $\tau : M\rightarrow N$ and $\pi : N\rightarrow M$ such that for each $S\in\mF\cap F$:
\begin{align*}
\tau(F/S)\ :\ M(F/S)& \rightarrow N(F/S)\\
(f:G/S{\buildrel{x}\over{\to}}G/H)&\mapsto (\tau(f):G/S{\buildrel{x}\over{\to}} G/H\cap F^x)\\
\pi(F/S)\ :\ N(F/S)& \rightarrow  M(F/S)\\
 (g:G/S{\buildrel{y}\over{\to}} G/H\cap F^x)&\mapsto(\pi(g):G/S{\buildrel{y}\over{\to}}G/H).
\end{align*}
where $f\in {\Z}[\mbox{map}_G(G/S, G/H)], g\in {\Z}[\mbox{map}_G(G/S, G/H\cap F^x)]$. It is easily checked that $\pi\tau=1_M$.
\end{proof}

\begin{prop}{\label{algseq}}
{\normalfont Let $\mF$ be any given family of subgroups of $G$.
Let $R_{\ast}\twoheadrightarrow \underline{\Z}$ be a resolution in
$\orbmod$ such that, for each $i$, $R_i\cong
\bigoplus_{\sigma\in
\Sigma_i}{\mbox{Ind}}_{G_{\sigma}}^{G}\underline{\Z},$ where
$\Sigma_i$ is an index set and $G_{\sigma}$ is a subgroup of $G$.
Then, for modules $N$ in $G$-Mod$_\mF$ and $K, M$ in Mod$_\mF G$,
there are first quadrant spectral sequences:
$$\mathrm{E}_{p,q}^1(N)=\displaystyle{\bigoplus_{\sigma \in \Sigma_p}}
\mathrm{Tor}_q^{\mF\cap G_{\sigma}}({\mbox{Res}}_{G_{\sigma}}^{G}K, {\mbox{Res}}_{G_{\sigma}}^{G}N)\Longrightarrow
\mathrm{Tor}_{p+q}^{\mF}(K, N),$$
$$\mathrm{E}^{p,q}_1(M)=\displaystyle{\prod_{\sigma \in \Sigma_p}}
\mathrm{Ext}^q_{\mF\cap G_{\sigma}}({\mbox{Res}}_{G_{\sigma}}^{G}K, {\mbox{Res}}_{G_{\sigma}}^{G}M)\Longrightarrow
\mathrm{Ext}^{p+q}_{\mF}(K, M).$$}
\end{prop}
\begin{proof}We derive the cohomological spectral sequence and note that the homological part is similar.

 %Following Martinez-Perez's construction (see p. 164), let $\pi:\mbox{Or}_\mathfrak{T}G\to \mbox{Or}_\mF G$ be the trivial functor induced by the inclusion of the trivial class $\mathfrak{T}$ into $\mF$.

Let $Q_{\ast}\twoheadrightarrow K$ be a free resolution of the module $K$ in $\mbox{Mod}_{\mF}G$.  By filtering the cochain complex:
$${C}_{\ast, \ast\ast}=\mbox{Hom}_{\mF}(Q_{\ast}\otimes R_{\ast \ast}, M)$$
by rows and by columns, we respectively obtain the $\mathrm{E}_2$- and $\mathrm{E}_1$-term first quadrant spectral sequences:
\begin{align*}
^{I}\mathrm{E}^{p,q}_2(M)&=
\mathrm{H}_h^p\mathrm{H}_v^q({C}_{\ast, \ast\ast})\Longrightarrow
\mathrm{H}^{p+q}_{\mF}(\mbox{Tot}({C}_{\ast, \ast\ast})),\\
^{II}\mathrm{E}^{p,q}_1(M)&=
\mathrm{H}_h^q({C}_{\ast, \ast\ast})\Longrightarrow
\mathrm{H}^{p+q}_{\mF}(\mbox{Tot}({C}_{\ast, \ast\ast})).
\end{align*}
%Since the augmented $\orb$-chain complex ${C}^{\mF}_{\ast \ast}(X, \underline{\Z})\twoheadrightarrow \underline{\Z}$ is exact and $Q_{p}$ is projective, $Q_{p}\otimes {C}^{\mF}_{\ast \ast}(X, \underline{\Z})\twoheadrightarrow Q_{p}$ is exact  for each $p$.
We claim that, for all $q>0$ and $p\geq 0$,  $^{I}\mathrm{E}^{p,q}_2(M)=0$  and $^{I}\mathrm{E}^{p,0}_2(M)=\mathrm{H}^{p}(\mathrm{Hom}_{\mF}(Q_{\ast}, M))=\mathrm{Ext}^{p}_{\mF}(K, M)$. Therefore, $\mathrm{H}^{p+q}_{\mF}(\mbox{Tot}({C}_{\ast, \ast\ast}))\cong \mathrm{Ext}^{p+q}_{\mF}(K, M)$ for all $p, q\geq 0$.

To this end, suppose $Q_{i}=\mathbb{Z}[\mbox{map}_G(-,G/F)]$ for some $G/F\in \orb$ and $i\geq 0$. Then,
\begin{align*}C_{i,\ast \ast}&=\mbox{Hom}_{\mF}(\mathbb{Z}[\mbox{map}_G(-,G/F)]\otimes R_{\ast \ast}, M)\\
&\cong\mbox{Hom}_{\mF}(\mbox{Ind}_{F}^{G}\underline{\Z}\otimes R_{\ast \ast}, M)\\
&\cong\mbox{Hom}_{\mF\cap F}(\mbox{Res}_{F}^{G}R_{\ast \ast}, \mbox{Res}_{F}^{G}M)\\
\end{align*}
By Lemma \ref{proj}, it follows that $\mbox{Res}_{F}^{G}R_{\ast \ast}$ is projective in $\mbox{Mod}_{\mF\cap F}F$. Thus, $\mbox{Res}_{F}^{G}R_{\ast \ast}\twoheadrightarrow \underline{\Z}$ is  a projective resolution in $\mbox{Mod}_{\mF\cap F}F$. Hence, $\mathrm{H}^q(C_{i,\ast \ast})=\mathrm{H}^q_{\mF\cap F}(F,\mbox{Res}_{F}^{G}M)=0$ for all $q>0$.

In general, since  $Q_i$ is a direct sum of free $\orb$-modules based at  objects of $\orb$, this argument shows that $\mathrm{H}^q_v(C_{\ast,\ast \ast})=0$ for $q>0$ and $\mathrm{H}^0_v(C_{\ast,\ast \ast})=\mathrm{Hom}_{\mF}(Q_{\ast}, M)$ which finishes the claim.

Also, by Lemma \ref{ind}, for each $p\geq 0$ we have:
\begin{align*}
\mbox{Hom}_{\mF}(Q_{\ast}\otimes R_{p}, M)&\cong \prod_{\sigma \in \Sigma_p}\mbox{Hom}_{\mF}(Q_{\ast}\otimes \mbox{Ind}_{G_{\sigma}}^{G}\underline{\Z}, M)\\
&\cong \prod_{\sigma \in \Sigma_p} \mbox{Hom}_{\mF\cap G_{\sigma}}(\mbox{Res}_{G_{\sigma}}^{G}Q_{\ast}, \mbox{Res}_{G_{\sigma}}^{G}M).
\end{align*}
Since restriction functor respects projective modules (see \cite[3.7]{mar}), this entails:
\begin{align*}^{II}\mathrm{E}^{p,q}_1(M)&=\prod_{\sigma \in \Sigma_p}\mathrm{H}^q( \mbox{Hom}_{\mF\cap G_{\sigma}}(\mbox{Res}_{G_{\sigma}}^{G}Q_{\ast}, \mbox{Res}_{G_{\sigma}}^{G}M))\\
&= \displaystyle{\prod_{\sigma \in \Sigma_p}}
\mathrm{Ext}^q_{\mF\cap G_{\sigma}}({\mbox{Res}}_{G_{\sigma}}^{G}K, {\mbox{Res}}_{G_{\sigma}}^{G}M).\end{align*}
By substituting the new terms into the $\mathrm{E}_1$-term spectral sequence, we obtain the desired spectral sequence.
\end{proof}

\begin{remk}{\label{algrem}}
{\normalfont When $K$ is the trivial module $\underline{\Z}$, we obtain the (co)homological spectral sequences:
$$\mathrm{E}_{p,q}^1(N)=\displaystyle{\bigoplus_{\sigma \in \Sigma_p}}
\mathrm{H}_q^{\mF\cap G_{\sigma}}(G_{\sigma}, {\mbox{Res}}_{G_{\sigma}}^{G}N)\Longrightarrow
\mathrm{H}_{p+q}^{\mF}(G, N),$$
$$\mathrm{E}^{p,q}_1(M)=\displaystyle{\prod_{\sigma \in \Sigma_p}}
\mathrm{H}^q_{\mF\cap G_{\sigma}}(G_{\sigma}, {\mbox{Res}}_{G_{\sigma}}^{G}M)\Longrightarrow
\mathrm{H}^{p+q}_{\mF}(G, M).$$}
\end{remk}

\begin{lem}{\label{sc:04.5}}{\normalfont Let $\mF$ be any given family of subgroups of $G$.  Suppose $X$ is a $G$-CW-complex and
denote by ${C}^{\mF}_{\ast\ast}(X, \underline{\Z})$  the Bredon chain complex of $X$. Then, for each $i$, we have:
$${C}^{\mF}_i(X, \underline{\Z})\cong \bigoplus_{\sigma\in \Sigma_i}{\mbox{Ind}}_{G_{\sigma}}^{G}\underline{\Z}.$$
where $\Sigma_i$ is a set of representatives of all the $G$-orbits of $i$-cells of $X$ and $G_{\sigma}$ is the isotropy of $\sigma$.}
\end{lem}
\begin{proof} For each $H\in \mF$, we obtain: $${C}^{\mF}_i(X, \underline{\Z})(G/H)= {C}_i(X^H, \Z)\cong \bigoplus_{\sigma\in \Sigma_i}\Z[\mbox{map}_G(G/H, G/G_{\sigma})].$$
On the other hand, $\Z[\mbox{map}_G(-, G/G_{\sigma})]\cong {\mbox{Ind}}_{G_{\sigma}}^{G}\underline{\Z},$ (see \cite[3.3]{mar}). Hence, the result is proved.
%
%Maybe, it is easier to prove the second statement by just considering $\underline{C^i}(X, M)$ as above and $\underline{\mbox{Coind}}_{G_{\sigma}}^{G}M$ as $\mbox{Hom}(R(??; F(?)); M)$ following the notation of Hambleton and Yalcin in ``Equivariant CW-complexes and the orbit category'' pages 6 (see also Lemma 3.1 of the same paper). We need to show that
%$$\mbox{Hom}(\underline{C_i}(X^?, \underline{\Z}), M(?))\cong \bigoplus_{\sigma\in \Sigma_i} \mbox{Hom}(\Z_{G_{\sigma}}(??; F_{G_{\sigma}}(?)); M),$$
%where $\Z_{G_{\sigma}}(??; F_{G_{\sigma}}(?)):{\mbox{Or}}(G_{\sigma})\times {\mbox{Or}}(G)\to \Z\mbox{-mod}$ is the bimodule map defined by $(x,y)\mapsto \mbox{Hom}_{\Z}(y,F_{G_{\sigma}}(x))$.
\end{proof}

\begin{thm}{\label{main}}{\normalfont Let $\mF$ be any given family of subgroups of $G$. Suppose $X$ is a $G$-CW-complex such that for each  $F\in \mF$ the subcomplex $X^F$ is contractible. Then, for modules $N$ in $G$-Mod$_\mF$ and $K, M$ in Mod$_\mF G$, there are first quadrant spectral sequences:
$$\mathrm{E}_{p,q}^1(N)=\displaystyle{\bigoplus_{\sigma \in \Sigma_p}}
\mathrm{Tor}_q^{\mF\cap G_{\sigma}}({\mbox{Res}}_{G_{\sigma}}^{G}K, {\mbox{Res}}_{G_{\sigma}}^{G}N)\Longrightarrow
\mathrm{Tor}_{p+q}^{\mF}(K, N),$$
$$\mathrm{E}^{p,q}_1(M)=\displaystyle{\prod_{\sigma \in \Sigma_p}}
\mathrm{Ext}^q_{\mF\cap G_{\sigma}}({\mbox{Res}}_{G_{\sigma}}^{G}K, {\mbox{Res}}_{G_{\sigma}}^{G}M)\Longrightarrow
\mathrm{Ext}^{p+q}_{\mF}(K, M),$$
where $\Sigma_p$ denotes
a set of representatives of all the $G$-orbits of $p$-cells of $X$.}
\end{thm}
\begin{proof} Let $R_{\ast}={C}^{\mF}_{\ast}(X, \underline{\Z})$ be the Bredon chain complex associated to $X$. By our hypothesis it is exact and thus a resolution of $\underline{\Z}$. The proof is now a direct consequence of Lemma \ref{sc:04.5} and  Proposition \ref{algseq}.
\end{proof}

\section{Applications}

Here, we present some  corollaries of the constructed spectral sequences.
Throughout this section, $cd_{\mF}(G)$  denotes the Bredon cohomological dimension of $G$ with respect to the family of subgroups $\mF$. We  substitute $cd_{\mF}(G)$ by either $\underline{cd}(G)$ or $\underline{\underline{cd}}(G)$ when $\mF$ is the family of  finite or  virtually cyclic subgroups, respectively.

As a first application of Theorem \ref{main}, we improve the bound obtained in \cite[3.4]{misl}.

\begin{cor}{\label{main5}}{\normalfont  Let $\mF$ be a family of subgroups of $G$. Suppose $X$ is a $G$-CW-complex such that for each  $F\in \mF$ the subcomplex $X^F$ is contractible. Suppose also there is an integer $b$ such that ${cd}_{\mF\cap G_v}(G_v)\leq b$ where $v$ runs over all the vertices of $X$.
Then, $${cd}_{\mF}(G)\leq b+\mbox{dim}(X).$$}
\end{cor}

%We note that our bound is optimal.
\medskip

%Let $gd_{\mF}(G)$ be the minimal dimension of all classifying spaces $E_{\mF}G$ of $G$ for the class $\mF$.
We obtain a similar result as in \cite[4.1]{mar} which can be seen as an algebraic analogue of a result of Luck and Weiermann  (see \cite[5.1]{LW}).

\begin{cor}{\label{main6}}{\normalfont  Let $G$ be a group and $\mF\subseteq \mG$ be families of subgroups of $G$. Suppose that there exists an integer $d$ such that for every $H \in \mG$ we have $cd_{\mF\cap H}(H)\leq d$. Then,
 $${cd}_{\mF}(G)\leq cd_{\mG}(G)+d.$$
In particular, ${\underline{cd}}(G)\leq {\underline{\underline {cd}}(G)}+1.$}
\end{cor}
\begin{proof}Let $R_{\ast}\twoheadrightarrow \underline{\Z}$ be the restriction to $\orbmod$ of a free resolution in $\mbox{Mod}_{\mG}G$ of length $cd_{\mG}(G)$. Note that $R_{\ast}\cong \bigoplus_{\sigma\in \Sigma_i}{\mbox{Ind}}_{G_{\sigma}}^{G}\underline{\Z},$
where $\Sigma_{\ast}$ is an index set and each $G_{\sigma}\in \mG$. Now, an easy application of the cohomological spectral sequence of \ref{algrem} for the family ${\mF}$ proves the first inequality.

From this and the fact that any virtually cyclic group has a 1-dimensional classifying space for proper actions we can deduce the second inequality.
\end{proof}

\begin{cor}{\label{sc:08}}{\normalfont Let $G$ be a countable directed union of groups $G_i$.  Let $\mF$ be a family of subgroups of $G$ such that
each $H\in \mF$ is a subgroup  of some $G_j$, Then,
$$cd_{\mF}(G)\leq 1+ \mbox{max}\{cd_{\mF\cap G_i}(G_i)\}.$$}
\end{cor}
\begin{proof} Since $G$ is countable, it follows by Bass-Serre theory \cite{ser} that there exists a $G$-tree $X$ such that $X^{G_i}\ne \emptyset$ for every $G_i$. Then, by Corollary \ref{main5}, it follows that $cd_{\mF}(G)\leq \mbox{dim}(X)+ \mbox{max}\{cd_{\mF\cap G_i}(G_i)\}$ as claimed.
\end{proof}
\begin{remk}{\rm Note that the above conditions on the family of subgroups are satisfied when $\mF$ is the class of finite subgroups, virtually cyclic subgroups, or more generally consisting of groups that are subgroups of finitely generated subgroups of $G$. See also a related result of Nucinkis in \cite[4.2]{nuc}.}\end{remk}

The following corollary is related to \cite[1.4]{misl} and gives a bound for $\underline{cd}(G)$.

\begin{cor}{\label{main7}}{\normalfont  Let $1\to H\to G{\buildrel\pi\over\longrightarrow} Q\to 1$ be a short exact sequence of groups and suppose that there is an upper bound $d$ on ${\underline{cd}}(\pi^{-1}(F))$ where $F$ ranges over all finite subgroups of $Q$. Then
$$\underline{cd}(G)\leq d+\underline{cd}(Q).$$}
\end{cor}
\begin{proof} Let $R_{\ast}\twoheadrightarrow \underline{\Z}$ be a free resolution in $\underline{\mbox{Mod}}Q$ of length $\underline{cd}(Q)$. Using the functor $\mathcal{F}:\underline{\mathcal{O}}G \to \underline{\mathcal{O}}Q$ induced by the projection $\pi:G\to Q$ we can inflate this resolution to a resolution $\bar R_{\ast}\twoheadrightarrow \underline{\Z}$ in $\underline{\mbox{Mod}}G$. We claim that each module $\bar R_{\ast}$ is a direct sum of induced modules ${\mbox{Ind}}_{G_{\sigma}}^{G}\underline{\Z}$ where each $G_{\sigma}= \pi^{-1}(F)$ for some finite subgroup $F$ of $Q$.

To see this, it is enough to show that, for each finite subgroup $F$ of $Q$, the free $\underline{\mathcal{O}}Q$-module $\Z[-,Q/F]$ is inflated to the $\underline{\mathcal{O}}G$-module $\Z[-,G/{\pi^{-1}(F)}]$, i.e. we need to construct a natural isomorphism between the functors $M=\Z[\mathcal{F}(-),Q/F]$ and  $N=\Z[-,G/{\pi^{-1}(F)}]$.
Let us define the natural transformations $\tau : M\rightarrow N$ and $\theta : N\rightarrow M$ such that for each finite subgroup $S$  of $G$:
\begin{align*}
\tau(G/S)\ :\ M(G/S)& \rightarrow N(G/S)\\
(Q/{\pi(S)}{\buildrel{x}\over{\to}}Q/F)&\mapsto (G/S\xrightarrow{\scriptscriptstyle \pi^{-1}(x)} G/{\pi^{-1}(F)}),\\
\theta(G/S)\ :\ N(G/S)& \rightarrow  M(G/S)\\
 (G/S{\buildrel{y}\over{\to}} G/{\pi^{-1}(F)})&\mapsto(Q/{\pi(S)}\xrightarrow{\scriptscriptstyle \pi(y)}Q/F).
\end{align*}
It readily follows that $\theta\tau=1_M$ and $\tau\theta=1_N$. This finishes the claim.

Now, by the spectral sequence of \ref{algrem}, we can conclude: $$\underline{cd}(G)\leq \mbox{sup}\{\underline{cd}(G_\sigma)\;|\; \sigma\in \Sigma_{\ast}\}+ \underline{cd}(Q)$$ and hence the result follows.
\end{proof}

Let $\underline{gd}(G)$ denote the minimal dimension of all classifying spaces $\underline{E}G$ of $G$. As a corollary, we deduce an algebraic analog of the following theorem of L\"uck (see \cite[3.1]{luck}).

\begin{cor}{\label{ext}}{\normalfont  Suppose $1\to H\to G{\buildrel\pi\over\longrightarrow} Q\to 1$ is a short exact sequence of groups and suppose further there is an upper bound $k$ on the order of the finite subgroups of $Q$. Then,
$$\underline{cd}(G)\leq k(\underline{gd}(H))+ \underline{cd}(Q).$$}
\end{cor}
\begin{proof} For each finite subgroup $F$ of $Q$, the subgroup $\pi^{-1}(F)$ is an extension of $H$ by $F$. By Serre's construction (see the proof of Proposition \ref{exten} or \cite{ser}), it has a $|F|(\underline{gd}(H))$-dimensional model for $\underline{E}\pi^{-1}(F)$. Hence, its  Bredon cohomological dimension is bounded by $k(\underline{gd}(H))$. Now, applying the previous corollary ends the proof.
\end{proof}

Next, we derive a long exact sequence which is the Bredon analogue of the long exact sequence in ordinary cohomology arising from a cellular action on a tree.

\begin{cor}{\label{les}}{\normalfont  Suppose $\mF$ is a family of subgroups of $G$ and $X$ is a $G$-tree such that for each  $F\in \mF$,  $X^F$ is nonempty and connected. Denote by $G_v$ the stabilizer of a vertex $v$ and by $G_e$ the stabilizer of an edge $e$. Let $\Sigma_0$ and $\Sigma_1$ be the sets of $G$-orbit representatives of the vertexes and the edges of $X$, respectively. Then, for every $M$  in Mod$_\mF G$, there is a long exact sequence:
{\small $$\cdots\to \mathrm{H}^{i}_{\mF}(G, M) \to\displaystyle{\bigoplus_{v \in \Sigma_0}} \mathrm{H}^i_{\mF\cap G_v}(G_{v}, {\mbox{Res}}_{G_{v}}^{G}M)\to \displaystyle{\bigoplus_{e \in \Sigma_1}} \mathrm{H}^i_{\mF\cap G_e}(G_{e}, {\mbox{Res}}_{G_{e}}^{G}M)\to \mathrm{H}^{i+1}_{\mF}(G, M) \to\cdots $$}}
\end{cor}
\begin{proof} Let us consider the spectral sequence of \ref{main} applied to the $G$-CW-complex $X$ and $K=\underline{\Z}$.  On the first page there are only two possibly nonzero vertical lines, i.e. $\mathrm{E}^{p,q}_1(M)=0$ for $p>1$. Therefore, we have a Wang sequence (see for example \cite[p.130]{weib}):
$$\cdots\to \mathrm{H}^{i}_{\mF}(G, M) \to \mathrm{E}^{0,i}_1(M)\to \mathrm{E}^{1,i}_1(M)\to \mathrm{H}^{i+1}_{\mF}(G, M) \to\cdots$$
Substituting the $\mathrm{E}_1$-terms, yields the desired long exact sequence.
\end{proof}

Consequently, we obtain a Mayer-Vietoris sequence for amalgamated products.

\begin{cor}{\label{MV}}{\normalfont  Let $G=A\ast_C B$, where $A$, $B$ and $C$ are groups and suppose $M$ in Mod$_\mF G$. Suppose, every  $F$ in $\mF$ is a subgroup of a conjugate of $A$ or $B$. Then, there is a long exact sequence:
{\small $$\cdots\to \mathrm{H}^{i}_{\mF}(G, M) \to \mathrm{H}^i_{\mF\cap A}(A, {\mbox{Res}}_{A}^{G}M)\oplus \mathrm{H}^i_{\mF\cap B}(B, {\mbox{Res}}_{B}^{G}M)\to \mathrm{H}^i_{\mF\cap C}(C, {\mbox{Res}}_{C}^{G}M)\to \mathrm{H}^{i+1}_{\mF}(G, M)\cdots $$}}
\end{cor}
\begin{proof} Let $X$ be the Bass-Serre tree of $G$. The result immediately follows from Corollary \ref{les}.
\end{proof}

\begin{remk}{\rm This Mayer-Vietoris sequence has already  been derived in \cite[3.32]{MV} by a different approach, using a homotopy-push-out of the induced spaces of the classifying spaces  corresponding to the subgroups in the amalgamated product.}\end{remk}

\section{\hix{}-groups}
We begin with the definition of a class of groups which is analogous to  Kropholler's definition of the class ${{\scriptstyle\bf H}}\mathfrak{X}$ (see \cite{krop}).

Unless otherwise specified, $\mF$ will denote a subgroup closed class of groups. Note that for any group $G$ the collection: $$\mF\cap G=\{H\leq G \; | \;  \mbox{$H$ is isomorphic to a subgroup in $\mF$}\}$$ is then a family of subgroups of $G$.   If $\mF$ is the class of finite groups, when referring to Bredon (co)homology of $G$ with respect to $\mF\cap G$, we will simply say Bredon (co)homology of $G$. We will assume throughout that whenever a group $G$ acts cellularly on a CW-complex, then the action of the stabilizer group of any cell fixes that cell pointwise.

%This is not a strong assumption on the CW-complex, because by a subdivision, we can assume that any $G$-CW-complex will satisfy this condition?

\begin{defn}{\label{sc:01}}{\normalfont Let $\mX$ be a class of groups. \hix{} is defined as the smallest class of groups containing the class
$\mX$ with the property that if a group $G$ acts cellularly on a
finite dimensional CW-complex $X$ with all isotropy
subgroups in \hix{} and such that for each subgroup $F\in \mF\cap G$ the fixed point set
$X^F$ is contractible, then $G$ is in \hix{}.}
\end{defn}

\begin{remk}{\rm When $\mF$ is the trivial class, then \hix{} is exactly the Kropholler's class  ${{\scriptstyle\bf H}}\mathfrak{X}$.
 Also, note that $X$ need not be a model for $E_{\mF}G$, because we do not require for $X^F$ to be empty if $F$ is not in $\mF$. It is, in a way, an intermediary complex.}\end{remk}
There is also an inductive definition of \hix{}-groups, by using ordinal
numbers:

\begin{enumerate}[(a)]
\item Let ${{\scriptstyle\bf H}}_{0}^\mathfrak{F}\mX=\mX$.

\smallskip

\item For ordinal $\beta >0$, define ${{\scriptstyle\bf H}}^\mathfrak{F}_{\beta}\mathfrak{X}$ to be the class of groups
that can act cellularly on a finite dimensional CW-complex $X$
such that each isotropy group is in ${{\scriptstyle\bf H}}^\mathfrak{F}_{\alpha}\mathfrak{X}$ for some $\alpha < \beta$ ($\alpha$
can depend on the isotropy group) and for each subgroup $F\in \mF\cap G$ the fixed point set
$X^F$ is contractible.
\end{enumerate}

We denote by $\scriptstyle\bf L\hix{}$ the class of locally \hix{}-groups, i.e. groups whose finitely generated subgroups belong to \hix{}.
\begin{remk}\rm {It immediately follows:
\begin{enumerate}[-]
\item a group is in \hix{} if and only if it belongs to  ${{\scriptstyle\bf H}}^\mathfrak{F}_{\alpha}\mathfrak{X}$ for some
ordinal $\alpha$;

\smallskip

\item  ${{\scriptstyle\bf H}}^\mathfrak{F}_{\alpha}\mathfrak{X}\subseteq {{\scriptstyle\bf H}}_{\alpha}\mathfrak{X}$ for each ordinal $\alpha$ where ${{\scriptstyle\bf H}}_{\alpha}\mathfrak{X}$ is the  Kropholler's class for ordinal $\alpha$;

\smallskip

\item $G$ admits a finite dimension model for $E_{\mF}G$ if and only if it is in ${{\scriptstyle\bf H}}^\mathfrak{F}_{1}\mF$.
\end{enumerate}}
\end{remk}

\begin{prop}{\label{exten}}{\normalfont   \hix{} is subgroup closed. If $\mF$  is the class of finite groups, then \hff{}\ is extension closed.}
\end{prop}

\begin{proof}  It is obvious that \hix{\alpha}\ is subgroup closed for every ordinal $\alpha$, so $\hix{}$ is subgroup closed.

The proof that \hff{}\ is extension closed is analogous to the proof of 2.3 in \cite{krop}. First we prove that  \hff{\alpha}\ is closed under finite extensions, using induction on the ordinal $\alpha$.  This is obvious for $\alpha=0$. Assume that it is true for all $\beta<\alpha$ and let $G$ be a group with a finite index subgroup $N$ that belongs to \hff{\alpha}. There is a finite dimensional $N$-CW-complex $X$ such that each isotropy group is in \hff{\beta} for some $\beta<\alpha$ and $X^F$ is contractible for any finite subgroup $F$ of $N$. Serre's construction (cf. \cite{brown}) shows that the $G$-set $Y=\mathrm{Hom}_G(N,X)$ has the structure of a CW-complex via the bijection $\varphi: Y\rightarrow\prod_{t\in T}X,$ with $\varphi(f)=(f(t))_{t\in T}$ where $T$ is a set of coset representatives of $N$ in $G$. If $K$ is the isotropy group of an element $f$ of $Y$ then $f(g)=f(gk)$ for all $k\in K$ and $g\in G$. In particular $f(e)=f(k)=kf(e)$ for all $k\in N\cap K$, so $N\cap K$ is a subgroup of the isotropy group of $f(e)$, which belongs to $\hff{\beta}$ for some $\beta<\alpha$. Since \hff{\beta} is subgroup closed it follows from the inductive hypothesis that $K$ belongs to \hff{\beta}. Next we show that $Y^F$ is contractible for every finite subgroup $F$ of $G$. For this, take a set $S$ of double coset representatives, such that $G=\cup_{s\in S}NsF$ and let $F_s=N\cap sFs^{-1}$. The map $\psi:Y^F\rightarrow\prod_{s\in S}X^{F_s}$ with $\psi(f)=(f(s))_{s\in S}$ is well defined because if $f\in Y^F$ and $g\in F_s$ we have $s^{-1}gs\in F$, so $gf(s)=f(gs)=f(ss^{-1}gs)=(s^{-1}gsf)(s)=f(s)$, thus $f(s)\in X^{F_s}$. Obviously $\psi$ is a bijection and each $X^{F_s}$ is contractible, so $Y^F$ is contractible.

Now, we prove by induction on the ordinal $\beta$ that if $N$ is an \hff{\alpha}-subgroup of $G$ such that $G/N$ is in \hff{\beta} then $G$ is in \hff{}. Indeed, the previous argument proves this claim for $\beta=0$.  Let $\beta>0$ and assume that the claim holds for $\gamma<\beta$. Let $X$ be a  $G/N$-CW-complex with stabilizers in \hff{\gamma} for $\gamma<\beta$ and $X^F$ is contractible for any finite subgroup $F$ of $G/N$. Then $G$ acts on $X$ via the projection $\pi:G\rightarrow G/N$ with cell stabilizers that are extensions of $N$ by \hff{\gamma}-groups for $\gamma<\beta$. It follows from the inductive hypothesis that these belong to \hff{}. Now, if $F$ is a finite subgroup of $G$ then $X^F=X^{\pi(F)}$ is contractible since $\pi(F)$ is a finite subgroup of $G/N$. Thus, $G$ belongs to \hff{}.
\end{proof}

\begin{prop}{\label{sc:02}}{\normalfont   If $\mF$  is the class of finite groups, then \hff{} is closed under countable directed unions, amalgamated products and HNN-extensions.}
\end{prop}
\begin{proof}
let $G=\bigcup_{n\in\mathbb N}G_n$ where $G_n\leq G_{n+1}$ and $G_n\in\hff{}$ for all $n\in\mathbb N$. Then it follows from Bass-Serre theory that $G$ acts on a tree $T$ such that $T^{G_n}\not=\emptyset$ for every $G_n$. If $F$ is a finite subgroup of $G$ then $F\leq G_n$ for some $n\in\mathbb N$, so $T^F\not=\emptyset$. Now, $T^F$
is connected, because if an element $g\in G$ stabilizes two vertices of $T$ then $g$ stabilizes the (unique) path that connects them. Thus $T^F$ is contractible.

If $G$ is an amalgamated product $A\ast_CB$ or an HNN-extension $K\ast_{\alpha}$, then it follows again from Bass-Serre theory that $G$ acts on a tree $T$ such that $T^A\not=\emptyset$, $T^B\not=\emptyset$ and $T^K\not=\emptyset$, respectively. The same is true for every conjugate of $A$, $B$ and $K$. Now,  any finite subgroup $F$ of $G$ is conjugate to a subgroup of $A$ or $B$, and $K$ respectively, so $T^F\not=\emptyset$ and repeating the above argument we have that $T^F$ is contractible.
\end{proof}

\begin{thm}{\label{sc:03}}{\normalfont If $\mF$ is the class of finite groups, then \hff{}\ contains all countable elementary amenable groups and all countable linear groups over a field of characteristic 0.}
\end{thm}

\begin{proof}
{\it Elementary amenable groups:} \hff{}\  contains all finitely generated abelian groups. Now, recalling Kropholler, Linnel and Moody's inductive definition of elementary amenable groups (see \cite{KLM}), by transfinite induction, using extensions and countable directed union closure, it follows that \hff{} contains all countable elementary amenable groups.

{\it Linear groups:} Since \hff{}\ is closed under countable directed unions, it suffices to consider finitely generated linear groups. In \cite[2.5]{krop}, Kropholler proves that countable linear groups over a field of characteristic 0 belong to \hf{}, making crucial use of results of Alperin and Shalen (see \cite{alperin}). Since torsion free \hf{}-groups belong to \hff{}, it follows that torsion free linear groups belong to \hff{}. Now, by a theorem of Borel (cf. Theorem 1.5 in \cite{alperin}) every finitely generated linear group is virtually torsion free. Since \hff{} is closed under finite extensions, the result follows.
\end{proof}

\section{Finiteness Properties}
A weaker notion than  (co)homological dimension called (co)homological jump height has been introduced in \cite{petro} and many of its finiteness properties  have been considered in \cite{petro2}. When studying group actions on certain complexes, we have seen that it is often more natural to utilize jump height rather than (co)homological dimension. Here, we extend the definition of jump (co)homology and jump height to Bredon cohomology and consider some applications.

\begin{defn}{\label{sc:04}}{\normalfont A discrete group $G$ has
{\it jump $\mF$-Bredon (co)homology} if there exists an
integer $k\geq 0$, such that for each subgroup $H$ of $G$ we have
$hd_{\mF\cap H}(H)=\infty$ ($cd_{\mF\cap H}(H)=\infty$) or $hd_{\mF\cap H}(H)\leq k$
($cd_{\mF\cap H}(H)\leq k$).  The minimum of all such $k$ is
 called {\it jump  height} and denoted $hjh_{\mF}(G)$ ($cjh_{\mF}(G)$).}
\end{defn}

First, we need a generalization of Benson's lemma (see \cite[5.6]{ben}) that applies in our context.

\begin{lem}{\label{benson}}{\normalfont Suppose that $G_{\alpha}$ ($0\leq \alpha<\gamma$,  $G_0=\{e\}$) is an ascending chain of groups with union $G=\cup_{\alpha<\gamma} G_{\alpha}$, for some ordinal $\gamma$. Let $\mF$ be a family of subgroups of $G$ such that every $H\in \mF$ is a subgroup of  some $G_{\alpha}$. If $M$ is a module in Mod$_\mF G$ whose restriction to each subgroup $G_{\alpha}$ is projective in Mod$_{\mF\cap G_{\alpha}} G_{\alpha}$, then $M$ has projective dimension at most one.}
\end{lem}
\begin{proof}Denote $Z_{G_{\alpha}}=\bigoplus_{\delta\geq\alpha}\Z[\mbox{map}_G(-, G/G_{\delta})]$.
Since $ M\otimes Z_{G_{\alpha}}\cong \bigoplus_{\delta\geq\alpha}\mbox{Ind}_{G_{\delta}}^{G}{\mbox{Res}}_{G_{\delta}}^{G}M$, by adjointness of induction and restriction and definition of projectives, we can deduce that $Q_{\alpha}:= M\otimes Z_{G_{\alpha}}$ is projective.

Given $\alpha\leq \beta$, we define an epimorphism $\psi_{\alpha}^{\beta}:Z_{G_{\alpha}}\to Z_{G_{\beta}}$ as follows: the summand $\Z[\mbox{map}_G(-, G/G_{\delta})]$ of $Z_{G_{\alpha}}$ is mapped to the summand $\Z[\mbox{map}_G(-, G/G_{\mu})]$ of $Z_{G_{\beta}}$, where $\mu=\mathrm{max}\{\beta, \delta\}$, by the natural quotient map $G/G_{\delta}\to G/G_{\mu}$. Note that $\psi_{\alpha}^{\beta}$ induces an epimorphism $id_M\otimes\psi_{\alpha}^{\beta}:Q_{\alpha}\to Q_{\beta}$.

Let $P_{\alpha}$ be the kernel of $id_M\otimes\psi_{0}^{\alpha}:Q_{0}\to Q_{\alpha}$. If $\alpha\leq\beta<\gamma$, then $P_{\alpha}\leq P_{\beta}$.  Since $Q_{\beta}$ is projective, the epimorphism $id_M\otimes\psi_{\alpha}^{\beta}:Q_{\alpha}\to Q_{\beta}$ must split. Then, the kernel which is isomorphic to $P_{\beta}/P_{\alpha}$, is projective.

Also, it follows from our hypothesis that for each $\alpha$, there exists an epimorphism from $Z_{G_{\alpha}}$ onto $\Z[\mbox{map}_G(-, G/G)]\cong \underline{\Z}$ that maps each summand  $\Z[\mbox{map}_G(-, G/G_{\delta})]$ of $Z_{G_{\alpha}}$ to  $\Z[\mbox{map}_G(-, G/G)]$ via the trivial quotient map $G/G_{\delta}\to G/G$.  This epimorphism induces an epimorphism $\phi_{\alpha}:Q_{\alpha}\to M$. Note that for each $\alpha\leq \beta$, we have $\phi_{\beta}\circ (id_M\otimes\psi_{\alpha}^{\beta})=\phi_{\alpha}$.

Let $P_{\gamma}$ the kernel of the induced epimorphism from $Q_{0}$ to $M$. The fact that $\varinjlim Z_{G_{\alpha}}=\underline{\Z}$, implies that $\varinjlim Q_{\alpha}=M$. So, we can deduce that $P_{\gamma}= \cup_{\alpha<\gamma} P_{\alpha}$. Now, by using transfinite induction, it is not difficult to see that,
$$P_{\gamma}\cong \bigoplus_{\alpha<\gamma} P_{\alpha+1}/P_{\alpha}$$ and hence it is projective. Since $M\cong P/P_{\gamma}$, it has projective dimension at most one.\end{proof}

Following the notation of the previous section, let $hd_{\mF\cap G}(G)$ ($cd_{\mF\cap G}(G)$) denote the Bredon (co)homological dimension of the group $G$ with respect to a class of groups $\mF$.

The next result is a generalization of Theorem 3.3 of \cite{petro}, Theorem 8 of \cite{tal}, and of Proposition 3.1 of \cite{JN}.

\begin{thm}{\label{sc:05}}{\normalfont Let $\mF$ be a subgroup closed class of groups and let $\mX$ be a class of groups with finite Bredon (co)homological dimension with respect to $\mF$.
Suppose $G$ is in ${{\scriptstyle\bf H}}^\mathfrak{F}\mX$ or $\mF\cap G$ consists of groups that are subgroups of finitely generated subgroups of $G$ and $G$ is in ${{\scriptstyle\mathrm{L}\bf H}}^\mathfrak{F}\mX$. Then $G$ has finite Bredon (co)homological dimension  if and only if $G$ has jump Bredon (co)homology with respect to $\mF$.}
\end{thm}
\begin{proof} We will prove the result for cohomology. The homological version is analogous.

One direction is trivial. So, let $G$  have jump Bredon cohomology with respect to $\mF$. Suppose $G$ is in \hix{}. Since $G$ is in ${{{\scriptstyle\bf H}}_{\alpha}^\mathfrak{F}\mX}$ for some ordinal $\alpha$, we can assume by transfinite induction, there exists a $G$-CW-complex $X$ such that for each  $F\in \mF\cap G$ the subcomplex $X^F$ is contractible and that each isotropy group has finite Bredon cohomological dimension bounded by ${cjh}_{\mF}(G)$. Now, applying Theorem \ref{main}, we can deduce that $\mathrm{H}^{i}_{\mF}(G, M)=0$ for all $i>{cjh}_{\mF}(G)+\mbox{dim}(X)$. Therefore, $G$ has finite Bredon cohomological dimension equal to ${cjh}_{\mF}(G)$.

Suppose $G$ is in ${{\scriptstyle\mathrm{L}\bf H}}^\mathfrak{F}\mX$ and $\mF$ consists of finitely generated groups. We can assume that $G$ is uncountable. Then, it can be expresses  as an ascending union of subgroups $G=\cup_{\alpha<\gamma} G_{\alpha}$, for some ordinal $\gamma$, such that each subgroup $G_{\alpha}$ has strictly smaller cardinality than $G$. Let $P_{\ast}\twoheadrightarrow \underline{\Z}$ be a projective resolution of $\underline{\Z}$ in Mod$_{\mF}G$ and denote $d={cjh}_{\mF}(G)$. According to Lemma \ref{benson}, the module $P_{d}$  has projective dimension at most 1. Hence, $\underline{\Z}$ has projective dimension at most $d+1$. It follows that $\cd(G)\leq d+1$, and hence $\cd(G)= d$.
\end{proof}

\begin{remk}{\label{sc:05.5}}{\normalfont We give an alternative proof of the first part of this result.

Suppose $G$ is countable. Let $X$ be a $G$-CW-complex  such that for each  $F\in \mF\cap G$ the subcomplex $X^F$ is contractible and that each isotropy group has finite Bredon cohomological dimension bounded by ${cjh}_{\mF}(G)$. Since $X^F$ is contractible for each $F\in \mF\cap G$, the chain complex ${C}^{\mF}_{\ast}(X, \underline{\Z})$ is exact.  By the induction, ${pl}_{\mF\cap G_{\sigma}}(\underline{\Z})\leq {cjh}_{\mF}(G)$ for each $G_{\sigma}$. Since  ${C}^{\mF}_{i}(X, \underline{\Z})\cong \bigoplus_{\sigma\in \Sigma_i}{\mbox{Ind}}_{G_{\sigma}}^{G}\underline{\Z}$, we obtain ${pl}_{\mF\cap G_{\sigma}}({C}^{\mF}_{i}(X, \underline{\Z}))\leq {cjh}_{\mF}(G)$ for each $i$. It follows that ${pl}_{\mF}(\underline{\Z})\leq {cjh}_{\mF}(G)+\mbox{dim}(X)$. Since $G$ has jump Bredon  cohomology, this shows ${pl}_{\mF}(\underline{\Z})= {cjh}_{\mF}(G)$.}
\end{remk}
%\begin{thm}{\label{sc:06}}{\rm Let $G$ be a group in ${{\scriptstyle\mathrm{L}\bf H}}^\mathfrak{F}\mF$, where $\mF$ is the class of finite groups. Then the following conditions are equivalent:
%\begin{enumerate} \item $G$ has a finite dimensional model for $\underline{E}G$; \smallskip \item the finitistic  Bredon  dimension of $G$ is finite; \smallskip \item $G$ has jump Bredon cohomology. \end{enumerate}}
%\end{thm}
%\begin{proof} It follows from Proposition \ref{fin} that (1)$\Rightarrow$ (2) and from Proposition \ref{finjump} that (2) $\Rightarrow$ (3). By Theorem \ref{sc:05} and L\"{u}ck's theorem \cite{luck1} we have that (3)$\Rightarrow$ (1)\end{proof}

\begin{remk}{\label{sc:05.6}} {\rm Note that by a result of L\"{u}ck \cite{luck1}, Theorem \ref{sc:05} establishes in particular that:
\begin{itemize}
\item[(i)] every ${{\scriptstyle\bf H}}^\mathfrak{F}\mF$-group  has jump $\mF$-Bredon cohomology if and only if $G$ has a finite dimensional model for ${E}_{\mF}G$.
\item[(ii)] every ${{\scriptstyle\mathrm{L}\bf H}}^\mathfrak{F}\mF$-group  has jump $\mF$-Bredon cohomology if and only if $G$ has a finite dimensional model for ${E}_{\mF}G$, where $\mF$ is the class of finite groups or the class of virtually cyclic groups.
\end{itemize}}
 \end{remk}
%We propose the following conjecture.
%
%\begin{conj}{\rm Let $G$ be an ${{\scriptstyle\mathrm{L}\bf H}}\mathfrak{F}$-group. Then $G$ has jump $\mF$-Bredon cohomology if and only if $G$ has a finite dimensional model for ${E}_{\mF}G$.}
%\end{conj}

 \begin{defn}{\rm  A group $G$ has periodic $\mathfrak F$-Bredon cohomology with period $q>0$ after $k$ steps, if the functors $\mathrm{H}^n_{\mF}(G,-)$ and $\mathrm{H}^{n+q}_{\mF}(G,-)$ are naturally equivalent for all $n>k$.}\end{defn}

Using the Eckmann-Shapiro lemma  it is easily proved that periodic Bredon cohomology after some steps is a subgroup closed property, i.e. every subgroup $S$ of $G$ has periodic $\mathfrak F\cap S$-Bredon cohomology with period $q$   after $k$ steps.
\begin{cor}{\rm  Let $G$ be in ${{\scriptstyle\mathrm{L}\bf H}}^\mathfrak{F}\mF$, where $\mF$ is the class of finite groups. Then $G$ has periodic Bredon cohomology after some steps if and only if $G$ admits a finite dimensional model for $\underline{E}G$.}\end{cor}
\begin{proof} Assume $G$ has periodic Bredon cohomology with period $q$ after $k$ steps, and let $S$ be a subgroup with $\underline{cd}(S)=r$. Then $\mathrm{H}^r_{\mF\cap S}(S,M)\not=0$ for some $\orbs$-module $M$. If $r>k$ then $\mathrm{H}^{r+q}_{\mF\cap S}(S,M)\cong\mathrm{H}^r_{\mF\cap S}(S,M)\not=0$ a contradiction, since $\mathrm{H}^{r+q}_{\mF\cap S}(S,-)=0$. It follows that $r\leq k$, so $G$ has jump Bredon cohomology, and applying Theorem \ref{sc:05} and L\"{u}ck's theorem \cite{luck1} we have that $G$ admits a finite dimensional model for $\underline{E}G$.

The converse is trivial, since $\mathrm{H}^n_{\mF}(G,-)= 0=\mathrm{H}^{n+q}_{\mF}(G,-)$ for any $q>0$ and $n>\underline{cd}(G)$. \end{proof}

\section{Applications for elementary amenable groups}
Lastly, we  derive some corollaries of Theorem \ref{sc:05} regarding elementary amenable groups.

\begin{prop}{\label{sc:09}}{\normalfont Let $\mathbb F$ be a field of characteristic zero and suppose $G$ is an elementary amenable group. Then, the conditions that $G$ has finite Hirsch length, $G$ has finite homological dimension over $\mathbb F$, $G$ has jump homology over $\mathbb F$, $G$ has jump Bredon homology, and $G$ has finite Bredon homological dimension are equivalent. Moreover,
$$h(G)=hd_{\mathbb F}(G)=hjh_{\mathbb F}(G)=\underline{hd}(G)=\underline{hjh}(G).$$}
\end{prop}
\begin{proof} By the analogues of Proposition 4.3 of \cite{nuc} and Lemma 2 of \cite{hill}  (see \cite[p.5]{FN}), $\underline{hd}(G)\geq hd_{\mathbb F}(G)\geq h(G)$.  But, by Theorem 1 of Flores and Nucinkis in \cite{FN}, $h(G)=\underline{hd}(G)$.

The second equality  is a consequence of Theorem \ref{sc:05} assuming $\mF$ is the trivial class, $\mX$ is the class of finite groups and the homology is considered over the field $\mathbb F$. The last equality follows again from Theorem \ref{sc:05} where $\mX=\mF$ is the class of finite groups.
\end{proof}

The next theorem is a generalization of Corollary 4.3 of \cite{JN}.

\begin{prop}{\label{sc:10}}{\normalfont
Suppose $\mathbb F$ be a field of characteristic zero and $G$ is an elementary amenable group. Let $X$ be
an $n$-dimensional $G$-CW-complex. Suppose there exists an integer
$k$ such that for
each $i > k$, $\mathrm{H}_i(X, \mathbb F)=0$  and $\mathrm{H}_k(X, \mathbb Z)\cong {\mathbb Z}^m \oplus A$ where $m>0$
and  $A$ is a finite group.
If all the stabilizer subgroups of the action of $G$ on $X$ have finite Bredon homological dimension
uniformly bounded by an integer $b$, then
$$\underline{hd}(G)\leq b + n - k +\displaystyle { {1\over 2}{m(m-1)}}.$$}
\end{prop}

\begin{proof} If $G$ is elementary amenable, by Proposition 5.1 of \cite{petro2} and Proposition \ref{sc:09}, $hjh_{\mathbb F}(G)< b + n - k +\displaystyle { {1\over 2}{m(m-1)}}$. Again applying \ref{sc:09}, finishes the proof.
\end{proof}

\begin{remk}{\rm Note that the assumptions of Theorem \ref{sc:10} can occur quite naturally; for instance -- when $X$ is a proper $G$-CW-complex which is finitely dominated or homotopy equivalent to a closed manifold. In particular, we can deduce the following.}\end{remk}

\begin{thm}{\label{sc:11}}{\rm Suppose $G$ is an elementary amenable group acting freely and properly discontinuously on an  $n$-dimensional manifold $M$ homotopy equivalent to a closed  manifold $N$ of dimension $k$. Then $\underline{cd}(G)\leq n-k+1$.}
\end{thm}
\begin{proof} We first argue that we can, without loss of generality, assume that the closed manifold $N$ is orientable. Indeed, suppose it is not and consider the oriented double covering $\overline{N}$ of $N$. Then, its pullback by the homotopy equivalence from $M$ to $N$ gives us a double covering $\overline{M}$ of $M$ and a homotopy equivalence from $\overline{M}$ to $\overline{N}$. Since, $G$ acts freely on $M$, we can deduce that there exists group $\overline{G}$ acting freely on $\overline{M}$ such that $\Z_2\unlhd \overline{G}$ and $\overline{G}/\Z_2\cong G$. But, in this case, $\underline{cd}(G)=\underline{cd}(\overline{G})$ (see \cite[5.5]{nuc}). So, we can assume that $N$ is orientable.

Now, since $\mathrm{H}_k(M, \Z)= \Z$ and $\mathrm{H}_i(M, \mathbb F)=0$, where $\mathbb F$ is a field of characteristic 0, we have that $\underline{hd}(G)\leq n - k$. Hence, $\underline{cd}(G)\leq n - k+1$, because $G$ is countable (see \cite[4.1]{nuc}).
%If $n-k+1\geq 3$, then by a theorem of L\"{u}ck (see \cite[13.19]{luck1}) it follows that $\underline{gd}(G)=\underline{cd}(G)$.
%
%Suppose $n-k\leq 1$. We will argue that this implies that $G$ is virtually cyclic and thus has geometric dimension 1. If
\end{proof}

%A group $G$ is said to have {\it periodic cohomology} with period $q$ after $k$-steps if the functors $\mathrm{H}^i(G, -)$ and $\mathrm{H}^{i+q}(G,-)$ are naturally equivalent for all $i>k$ (see \cite{tal2}).

%In finishing,  we give an analog of \cite[Theorem 9(3)]{tal} where now the boundedness condition on the order of finite subgroup of $G$ is no longer needed.

%\begin{cor}{\rm Let $G$ be an elementary amenable group with periodic cohomology after some steps. Then $G$ has a finite Bredon homological dimension.} \end{cor}
%\begin{proof} By Theorem 9(5) of \cite{tal}, there exists a finite dimensional free $G$-CW-complex $X$ homotopy equivalent to a sphere. Thus, applying Theorem \ref{sc:10} finishes the argument. \end{proof}

\end{document}